\documentclass[preprint,12pt]{elsarticle}

\usepackage{amssymb}

\usepackage{amsmath}
  \usepackage{paralist}
  \usepackage{graphics} %% add this and next lines if pictures should be in esp format
  \usepackage{epsfig} %For pictures: screened artwork should be set up with an 85 or 100 line screen
 \usepackage[colorlinks=true]{hyperref}

\allowdisplaybreaks[4]

\usepackage{graphicx,amssymb,mathrsfs,latexsym,amsthm}
\usepackage{verbatim}

\newtheorem{theorem}{Theorem}
\newtheorem{lemma}[theorem]{Lemma}
\newtheorem{corollary}[theorem]{Corollary}

\newtheorem{remark}[theorem]{Remark}

\begin{document}

\begin{frontmatter}

%% Title, authors and addresses

%% use the tnoteref command within \title for footnotes;
%% use the tnotetext command for theassociated footnote;
%% use the fnref command within \author or \address for footnotes;
%% use the fntext command for theassociated footnote;
%% use the corref command within \author for corresponding author footnotes;
%% use the cortext command for theassociated footnote;
%% use the ead command for the email address,
%% and the form \ead[url] for the home page:
%% \title{Title\tnoteref{label1}}
%% \tnotetext[label1]{}
%% \author{Name\corref{cor1}\fnref{label2}}
%% \ead{email address}
%% \ead[url]{home page}
%% \fntext[label2]{}
%% \cortext[cor1]{}
%% \address{Address\fnref{label3}}
%% \fntext[label3]{}

\title{Variational principle for subadditive sequence of potentials in bundle RDS}

%% use optional labels to link authors explicitly to addresses:
%% \author[label1,label2]{}
%% \address[label1]{}
%% \address[label2]{}

\author[Shanghai,Nanjing]{Xianfeng Ma}
\ead{xianfengma@gmail.com}
\author[Nanjing,NJU]{Ercai Chen}
\ead{ecchen@njnu.edu.cn}
\address[Shanghai]{Department of Mathematics, East China University of Science and
Technology\\ Shanghai 200237, China}
\address[Nanjing]{School of Mathematics and Computer Science, Nanjing Normal
University\\Nanjing 210097, China}
\address[NJU]{Center of Nonlinear Science, Nanjing University\\Nanjing 210093,
China}

\begin{abstract}
%% Text of abstract
The topological pressure is defined for  subadditive sequence of
potentials  in  bundle random dynamical systems. A variational
principle for the topological pressure is set up in a very weak
condition. The result may have some applications in the study of
multifractal analysis for random version  of nonconformal dynamical
systems.
\end{abstract}

\begin{keyword}
%% keywords here, in the form: keyword \sep keyword
Random dynamical system \sep subadditive sequence of potentials \sep
variational principle
%% PACS codes here, in the form: \PACS code \sep code

%% MSC codes here, in the form: \MSC code \sep code
%% or \MSC[2008] code \sep code (2000 is the default)
\MSC 37D35 \sep 37A35 \sep 37H99
\end{keyword}

\end{frontmatter}

%% \linenumbers

%% main text
\section{Introduction}
\label{Introduction} The topological pressure for single potential
was first presented by Ruelle \cite{Ruelle} for expansive maps.
Walters \cite{Walters1975} generalized it to general continuous
maps. The theory about the topological pressure, variational
principle and equilibrium states plays a fundamental role in
statistical mechanics, ergodic theory and dynamical systems
(\cite{Denker1976,Bowen79,Ruelle1978,Bowen,Pesin,Denker}). Falconer
\cite{Falconer} introduced the topological pressure for subadditive
sequence of potentials on mixing repellers. Cao \cite{Cao} extended
this notion to general compact dynamical systems. The topological
pressure for nonadditive sequence of potentials has proved valuable
tool in the study of multifractal formalism of dimension theory,
especially for nonconformal dynamical systems
\cite{Falconer,Barreira2006,Cao2008}.

In random dynamical systems (RDS), the topological pressure is also
important in the study of chaotic properties of random
transformations
\cite{Denker2008,Kifer1996,Bogen1995,Kifer1992,Kifer2008}. The
earlier work on the topological pressure for single potential was
due to Ledrappier \cite{LedWal} and Bogensch{\"u}tz \cite{Bogen}.
Bogensch{\"u}tz \cite{Bogen1999} also established the random version
of the Bowen-Ruelle formula for expanding almost conformal bundle
RDS. Later, Kifer \cite{Kifer2001} generalized this notion to
general bundle RDS and set up the corresponding variational
principle. Thus it is a natural question if there exists a random
version of thermodynamic formalism for subadditive sequence of
potentials, which probably have some potential applications in the
study of multifractal formalism of nonconformal RDS.

In this paper, we give the definition of  topological pressure for
subadditive sequence of potentials and derive a variational
principle for the topological pressure.  In fact, we formulate a
variational principle between the topological pressure,
measure-theoretic entropies of RDS and some functions about the
invariant measure. Our conditions for this principle are very weak.
We only assume that the topological pressure is not $-\infty$. As to
the case of $-\infty$, the condition $\Phi^*(\mu)=-\infty$ for all
invariant measure $\mu$ is equivalent to that the topological
pressure is $-\infty$. The result generalizes both Kifer's additive
variational principle to subadditive case and Cao's result in
deterministic dynamical systems to bundle RDS. The method we used is
still in the framework of Kifer's approach \cite{Kifer2001}, which
is the generalization of Misiurewicz's elegant proof of the
nonadditive variational principle \cite{Misiurewicz}. However, since
the technique for tackling the subadditive sequence is different
from the additive case, we make some changes.

This paper is organized as follows. In section \ref{TopPre} we give
a short description of the definitions of bundle RDS, the
measure-theoretic entropies of RDS and the topological pressure for
the subadditive sequence of potentials together with a corollary. In
section \ref{SubVP}, we state the variational principle for the
subadditive sequence of potentials and give the proof. A required
lemma is also given there.

\section{Preliminary}
\label{TopPre} Let $(\Omega,\mathcal {F}, \mathbf{P})$ be a
probability space together with an invertible
$\mathbf{P}$-preserving transformation $\vartheta$, where
$\mathcal{F}$ is assumed to be complete, countably generated and to
separate points. Let $(X,d)$ be a compact metric space together with
the Borel $\sigma$-algebra $\mathcal{B}$. A set $\mathcal{E}\subset
\Omega \times X$ is measurable with respect to the product
$\sigma$-algebra $\mathcal{F}\times \mathcal{B}$ and such that the
fibers $\mathcal{E}_{\omega}=\{ x\in X: (\omega, x)\in
\mathcal{E}\}$, $\omega \in \Omega$, are compact. A continuous
bundle random dynamical system (RDS) over $(\Omega,\mathcal {F},
\mathbf{P}, \vartheta)$ is generated by map $T_{\omega}:
\mathcal{E_{\omega}}\rightarrow \mathcal{E_{\vartheta \omega}}$ with
iterates $T_{\omega}^{n}=T_{\vartheta^{n-1}\omega} \cdots
T_{\vartheta\omega}T_{\omega}$, $ n \geq 1$, so that the map
$(\omega, x) \rightarrow  T_{\omega}x$ is measurable and the map $x
\rightarrow T_{\omega}x$ is continuous for $\mathbf{P}$-almost all
(a.a) $\omega$. The map $\Theta : \mathcal{E} \rightarrow
\mathcal{E}$ defined by $\Theta(\omega,x)=(\vartheta\omega,
T_{\omega}x)$ is called the skew product transformation.

Let $\mathcal {P}_{\mathbf{P}}(\mathcal{E}) = \{\mu \in \mathcal
{P}_{\mathbf{P}}(\Omega \times X) : \mu(\mathcal{E})=1\}$, where
$\mathcal {P}_{\mathbf{P}}(\Omega \times X)$ is the space of
probability measures on $\Omega \times X$ having the marginal
$\mathbf{P}$ on $\Omega$. Any $\mu \in \mathcal
{P}_{\mathbf{P}}(\mathcal{E})$ on $\mathcal{E}$ can be disintegrated
as $d \mu(\omega, x)=d \mu_{\omega}(x)\,d\mathbf{P}(\omega)$ (See
\cite{Dudley}), where $\mu_{\omega}$ are regular conditional
probabilities with respect to the $\sigma$-algebra $\mathcal
{F}_{\mathcal{E}}$ formed by all sets $(A\times X)\cap \mathcal{E}$
with $A \in \mathcal{F}$. Let $\mathcal
{M}_{\mathbf{P}}^1(\mathcal{E}, T)$ be the set of $\Theta$-invariant
measures $\mu \in\mathcal {P}_{\mathbf{P}}(\mathcal{E})$. $\mu$ is
$\Theta$-invariant if and only if the disintegrations $\mu_{\omega}$
of $\mu$ satisfy $T_{\omega}\mu_{\omega}=\mu_{\vartheta\omega}\,
\mathbf{P}$-a.s. \cite{Arnold}. Let $\mathcal {Q}=
\{\mathcal{Q}_{i}\}$ be a finite measurable partition of
$\mathcal{E}$, and $\mathcal {Q}(\omega)=
\{\mathcal{Q}_{i}(\omega)\}$, where
$\{\mathcal{Q}_{i}(\omega)\}=\{x\in \mathcal{E}_{\omega}:
(\omega,x)\in \mathcal{Q}_{i}\}$, is a partition of
$\mathcal{E}_{\omega}$. The conditional entropy of $\mathcal{Q}$
given the $\sigma$-algebra $\mathcal {F}_{\mathcal{E}}$ is defined
by
\begin{equation}
\label{Con-entropy}
 H_{\mu}(\mathcal{Q}\mid \mathcal {F}_{\mathcal{E}}) = -\int \sum_i
\mu(\mathcal{Q}_i \mid \mathcal{F}_{\mathcal{E}}) \log
\mu(\mathcal{Q}_i \mid \mathcal{F}_{\mathcal{E}})\,d\mathbf{P} =
\int H_{\mu_{\omega}}(\mathcal{Q}(\omega))\,d\mathbf{P}(\omega),
\end{equation}
where $H\mu_{\omega}(\mathcal{A})$ denotes the usual entropy of a
partition $\mathcal{A}$.
 The entropy
$h_{\mu}^{(r)}(T)$ of the RDS $T$ with respect to $\mu$ is defined
by the formula
\begin {equation}
\label{entropy}
 h_{\mu}^{(r)}(T)=\sup_{\mathcal{Q}}h_{\mu}^{(r)}(T,
\mathcal{Q}), \quad \text{where} \quad
 h_{\mu}^{(r)}(T, \mathcal{Q})=
\lim_{n\rightarrow \infty}\frac{1}{n} H_{\mu} \,\Bigg(
\bigvee_{i=0}^{n-1}(\Theta^i)^{-1}\mathcal{Q}\mid
\mathcal{F}_{\mathcal{E}}\Bigg),
\end {equation}
the supremum is taken over all finite measurable partitions
$\mathcal{Q}=\{\mathcal{Q}_i\}$ of $\mathcal{E}$ with finite
conditional entropy $H_{\mu}(\mathcal{Q}\mid
\mathcal{F}_{\mathcal{E}})< \infty$. It should be noted that the
supremum can be taken only over partitions $\mathcal{Q}$ of
$\mathcal{E}$ into sets $Q_i$ of the form $Q_i=(\Omega\times
P_i)\cap\mathcal{E}$, where $\mathcal{P}=\{P_i\}$ is a partition of
$X$ into measurable sets, so that
$Q_i(\omega)=P_i\cap\mathcal{E}_{\omega}$ (See
\cite{Bogen,Kifer,Bogenthesis}). By \eqref{Con-entropy}, the limit
can be also expressed as
\begin{equation}
h_{\mu}^{(r)}(T, \mathcal{Q})= \lim_{n\rightarrow \infty}\frac{1}{n}
\int H_{\mu_{\omega}}\,\Bigg(\bigvee_{i=0}^{n-1}(T_{\omega}^i)^{-1}
\mathcal{Q}(\vartheta^i\omega)\Bigg)\,d\mathbf{P}(\omega).
\end{equation}

For each measurable in $(\omega, x)$ and continuous in $x \in
\mathcal{E}_{\omega}$ function $f$ on $\mathcal{E}$, let
$$
\| f \| =\int \|  f(\omega) \|_{\infty} \, d\mathbf{P}, \quad
\text{where} \quad \|  f(\omega) \|_{\infty} = \sup_{x\in
\mathcal{E}_{\omega}}\mid f(\omega,x) \mid,
$$
and $\mathbf{L}_{\omega}^1 (\Omega, \mathcal{C}(X))$ be the space of
such functions $f$ with $\|f \| < \infty$ and identify $f$ and $g$
provided $\| f-g\| =0$, then $\mathbf{L}_{\omega}^1 (\Omega,
\mathcal{C}(X))$ is a Banach space with the norm $\| \cdot \|$.

Let $\Phi =\{ f_n\}_{n=1}^{\infty}$ be a sequence functions on
$\mathcal{E}$ such that each $f_n$ is measurable in $\omega$ and
continuous in $x$  on $\mathcal{E}$. These functions are measurable
in $(\omega, x)$ in view of Lemma III.14 from \cite{Kifer}.
 $\Phi$ is called \textit{subadditive} if for any $(\omega, x)\in
 \mathcal{E}$ and $ m, n\in\mathbb{N}$,
$$
f_{n+m}(\omega, x)\leq f_n(\omega, x)+f_m (\Theta^n(\omega, x)).
$$
If $f_1\in \mathbf{L}_{\mathcal{E}}^1(\Omega, \mathcal{C}(X))$ and
the above inequality is satisfied, then a simple calculation
indicates that each $f_n \in \mathbf{L}_{\mathcal{E}}^1(\Omega,
\mathcal{C}(X))$. In the sequel we always assume $\Phi$ satisfies
these conditions.

For any $\Theta$-invariant measure $\mu$, denote
$$
\Phi^*(\mu)=\lim_{n\rightarrow \infty}\frac{1}{n}\int f_n \,d\mu.
$$
Existence of the limit follows from the well-known subadditive
argument. If we denote $\Phi^k=\{f_{kn}\}_{n=1}^{\infty}$ for any
$k\in \mathbb{N}$, then $(\Phi^k)^*(\mu)= k \Phi^*(\mu)$.

For each $n\in\mathbb{N}$  and $\epsilon >0$, a family of metrics
$d_n^{\omega}$ on $\mathcal{E}_{\omega}$ is defined as
$$d_n^{\omega}(x,y)=\max_{0 \leq k < n}(d(T_{\omega}^k x,
T_{\omega}^k y)), \quad  x, y \in \mathcal{E}_{\omega},$$ where
$T_{\omega}^0$ is the identity map.
 For each $n\in\mathbb{N}$  and
$\epsilon >0$, a set $F\subset \mathcal{E}_{\omega}$ is said to be
$(\omega, \epsilon, n)$-separated if $x, y \in F$, $x\neq y$ implies
$d_n^{\omega}(x,y)>\epsilon$.

For $\Phi= \{f_n\}_{n=1}^{\infty}$, $\epsilon>0$ and an $(\omega,
\epsilon, n)$-separated set $F\subset \mathcal{E}_{\omega}$, denote
$$\pi_T(\Phi)(\omega, \epsilon, n
)=\sup\{\sum_{x\in F}\exp(f_n(\omega,x)): F\, \text{is an} (\omega,
\epsilon, n )\text{-separated subset of } \mathcal{E}_{\omega} \}.
$$
Obviously, the supremum  can be taken only over all maximal
$(\omega, \epsilon, n )$-separated subsets. By replacing the
function $S_n f$ in Lemma 1.2 of \cite{Kifer2001} with $f_n$, a
completely similar argument can give the following result, which
provides basic measurable properties we needed. In fact, for any
measurable function $g$ on the $\mathcal{E}_{\omega}$, the result is
also correct.

\begin{lemma}\label{lemkifer}
For any $n\in \mathbb{N}$ and $\epsilon > 0$ the function
$\pi_T(\Phi)(\omega, \epsilon, n )$is measurable in $\omega$, and
for each $\delta > 0$ there exists a family of maximal $(\omega,
\epsilon, n )$-separated set $G_{\omega} \subset
\mathcal{E}_{\omega}$ satisfying
\begin{equation}
\sum_{x\in G_{\omega}}\exp(f_n(\omega,x))\geq (1-\delta)
\pi_T(\Phi)(\omega, \epsilon, n )
\end{equation}
and depending measurably on $\omega$ in the sense that $G=\{
(\omega,x):x\in G_{\omega}) \} \in \mathcal{F} \times \mathcal {B}$,
In particular, the supremum in the definition of
$\pi_T(\Phi)(\omega, \epsilon, n )$ can be taken only measurable in
$\omega$ families of $(\omega, \epsilon, n )$-separated sets.
\end{lemma}

In view of this lemma, for $\Phi= \{f_n\}_{n=1}^{\infty}$ as above
and $\epsilon >0$, we can denote
\begin{equation}
\label{pressure} \pi_T(\Phi)(\epsilon)= \limsup_{n\rightarrow
\infty} \frac{1}{n} \int \log \pi_T(\Phi)(\omega,\epsilon, n) \,
d\mathbf{P}(\omega).
\end{equation}
The topological pressure of $\Phi$ is defined as
$$
\pi_T(\Phi)=\lim_{\epsilon\rightarrow 0}\pi_T(\Phi)(\epsilon),
$$
since $\pi_T(\Phi)(\epsilon)$ is a monotone decreasing function in
$\omega$, the limit exists and  the limit in fact equals to
$\sup_{\epsilon>0}\pi_T(\Phi)(\epsilon)$.

For a given $n\in\mathbb{N}_{+}$, through replacing $\vartheta$ by
$\vartheta^n$ we can consider the bundle RDS $T^k$ defined by
$(T^k)_{\omega}^n=T_{\vartheta^{(n-1)k}\omega}^k \cdots
T_{\vartheta^k \omega}^k T_{\omega}^k$.

\begin{corollary}\label{lem1}
If $\Phi= \{f_n\}_{n=1}^{\infty}$ is a subadditive sequence of
functions, each $f_n$ is  measurable in $\omega$ and continuous in
$x$ on $\mathcal{E}$ and $f_1 \in \mathbf{L}_{\mathcal{E}}^1(\Omega,
\mathcal{C}(X))$, then for any $n \in \mathbb{N}_{+}$,
$\pi_{T^k}(\Phi^k) = k\pi_T(\Phi)$.
\end{corollary}

\begin{proof}
 Since each $(\omega, \epsilon, n)$-separated set for $T^k$ is also a
 $(\omega, \epsilon, kn)$ for $T$, then
 $\pi_T(\Phi)(\omega,$ $ \epsilon,kn)\geq \pi_{T^k}(\Phi^k)(\omega, \epsilon, n)$
 and $\pi_{T^k}(\Phi^k) \leq k\pi_T(\Phi)$ follows.
 For any $\epsilon >0$, by the continuity of $T_{\omega}$, there
 exists some small enough $\delta >0$ such that if $d(x,y)\leq
 \delta$, $x, y \in \mathcal{E}_{\omega}$ then
 $d_{\omega}^k(x,y)<\epsilon$. For any positive integer $m$, there
  exists some integer $n$ such that $kn \leq m < k(n+1)$. It is easy
  to see that any $(\omega, \epsilon, m)$-separated set of
  $T$ is also an $(\omega, \delta, n)$-separated set of $T^k$.
  In  view of $f_m(\omega, x) \leq f_{kn}(\omega,x) +
  f_{m-kn}(T^{kn}(\omega,x))$ and
  $f_{m-kn}(T^{kn}(\omega,x))\leq \sum_{i=k}^{m-1}f_1(T^i(\omega,x))$,
  we have
  \begin{align*}
  \pi_T(\Phi)(\omega, \epsilon, m)
  &=\sup\{ \sum_{x\in F}\exp f_m(\omega,x):
  F \, \text{is an} \, (\omega, \epsilon, m)\text{-separated set of } T \}\\
  \begin{split}
  \leq  \sup\{ \sum_{x\in F}\exp (f_{kn}(\omega,x)+
  \sum_{i=k}^{m-1}f_1(T^{i}(\omega,x))):\quad\quad\quad\quad\quad\quad\quad \\
  F \, \text{is an} \,  (\omega, \epsilon, m) \text{-separated set of } T\}
  \end{split}\\
  \begin{split}
  \leq \sum_{i=k}^{m-1}\parallel f_1(\vartheta^i \omega) \parallel_{\infty}
   \sup\{ \sum_{x\in F}\exp (f_{kn}(\omega,x):\quad\quad\quad\quad\quad\quad\quad\quad \\
  F \, \text{is an} \, (\omega, \delta, n) \text{-separated set of }
  T^k\}.
  \end{split}
  \end{align*}
 Since $f_1 \in \mathbf{L}_{\mathcal{E}}^1(\Omega,\mathcal{C}(X))$,
 so $\int \sum_{i=k}^{m-1}\parallel f_1(\vartheta^i \omega)
 \parallel_{\infty}  \,d\mathbf{P}(\omega) < \infty $,
  then by \eqref{pressure},
  $ k\pi_T(\Phi)(\epsilon) \leq \pi_{T^k}(\Phi^k)(\delta) $.
  If $\epsilon \rightarrow 0$, then $\delta \rightarrow 0$, so the
  inequality opposite follows from the definition of the topological
  pressure.
\end{proof}

In the argument of the variational principle, only the first part of
the corollary is used. However, for integrability, we give the other
part, which shows the similarity with the usual additive situation,
i.e., $f_n=\sum_{i=0}^{n-1}f_1{T^i(\omega,x)}$.

\section{The variational principle for subadditive sequence of potentials}

\label{SubVP} First we give the following Lemma which we need in the
proof of our main theorem.
\begin{lemma}\label{lem2}
 For a sequence probability measures $\{\mu_n\}_{n=1}^{\infty}$ in
 $\mathcal{P}_{\mathbf{P}}(\mathcal{E})$, where $\mu_n
 =\frac{1}{n}\sum_{i=0}^{n-1}\Theta^i\nu_n$ and $\{\nu_n\}_{n=1}^{\infty}\subset
 \mathcal{P}_{\mathbf{P}}(\mathcal{E})$, if $\{n_i\}$ is some
 subsequence of natural numbers $\mathbb{N}$ such that $\mu_{n_i}\rightarrow \mu \in
 \mathcal{M}_{\mathbf{P}}^1(\mathcal{E},T)$, then for any $k\in
 \mathbb{N}$,
\begin{equation}\label{seq1}
\limsup_{i\rightarrow \infty}\frac{1}{n_i} \int f_{n_i}(\omega,x) \,
d\nu_{n_i} \leq \frac{1}{k}\int f_k \,d\mu.
\end{equation}
In particular, the left part is no more than $\Phi^*(\mu)$.
\end{lemma}

\begin{proof}
For $0\leq j <k$ and $n\geq 2k$, by the subadditivity of $\Phi$,
\begin{align*}
f_n &\leq f_{n-j}\Theta^j + f_j \\
    &\leq (f_k + f_k \Theta^k + \cdots + f_k\Theta^{[\frac{n-j}{k}-1]k}
    + f_{n-j-[\frac{n-j}{k}]k}\Theta^{[\frac{n-j}{k}]k}) \Theta^j
    +f_j \\
    &= \sum_{l=0}^{[\frac{n-j}{k}-1]}f_k \Theta^{lk+j} +
    (f_{n-j-[\frac{n-j}{k}]k}\Theta^{[\frac{n-j}{k}]k+j} +f_j)
\end{align*}
Summing over $j$ and dividing by $k$, we have
\begin{align*}
f_n  &\leq \frac{1}{k}
\sum_{j=0}^{k-1}\sum_{l=0}^{[\frac{n-j}{k}-1]}f_k \Theta^{lk+j} +
\frac{1}{k}
\sum_{j=0}^{k-1}(f_{n-j-[\frac{n-j}{k}]k}\Theta^{[\frac{n-j}{k}]k+j}
    +f_j) \\
      &=\frac{1}{k} \sum_{s=0}^{n-k}f_k\Theta^s +
       \frac{1}{k}
\sum_{j=0}^{k-1}(f_{n-j-[\frac{n-j}{k}]k}\Theta^{[\frac{n-j}{k}]k+j}
    +f_j).
\end{align*}
Since $n-j-[\frac{n-j}{k}]\leq k$ and $f_r\leq
\sum_{t=0}^{r-1}f_1\Theta^t$, where $ 1\leq r\leq k$, integrating
this  inequality, then by $f_1\in \mathbf{L}_{\mathcal{E}}^1
(\Omega,\mathcal{C}(X))$, we get

\begin{align}\label{seqlem2}
\int f_n \, d\nu_n
  &\leq \frac{1}{k}\int \sum_{s=0}^{n-k}f_k\Theta^s \,d\nu_n +
       \frac{1}{k} \int
\sum_{j=0}^{k-1}(f_{n-j-[\frac{n-j}{k}]k}\Theta^{[\frac{n-j}{k}]k+j}
    +f_j)d\nu_n \nonumber \\
  & \leq \frac{1}{k}\int \sum_{s=0}^{n-k}f_k\Theta^s \,d\nu_n +
\frac{1}{k} 2k^2 \int  \|f(\omega) \|_{\infty}
\,d\mathbf{P}(\omega) \nonumber \\
  & = \frac{n-k+1}{k}\int f_k \,d\mu^{\prime}_n + 2k \int  \|f(\omega) \|_{\infty}
\,d\mathbf{P}(\omega),
\end{align}
where $\mu^{\prime}_n=\frac{1}{n-k+1}\sum_{s=0}^{n-k}\Theta^s\nu_n$.
Since for any $f\in
\mathbf{L}_{\mathcal{E}}^1(\Omega,\mathcal{C}(X))$,
\begin{align*}
&n\int f\,d\mu_n -(n-k+1)\int f\, d\mu^{\prime}_n \\
=&\sum_{i=n-k+1}^{n-1}\int f\Theta^i \, d\nu_n \leq  k\int
\|f(\omega)\|_{\infty} \,d\mathbf{P}(\omega).
\end{align*}
Dividing by $n$ and letting $n\rightarrow \infty$, we get
$$
\lim_{n\rightarrow \infty}\int f \,d\mu_n = \lim_{n\rightarrow
\infty}\int f \,d\mu_n^{\prime}.
$$
Observing that $\lim_{i\rightarrow \infty}\mu^{\prime}_{n_i}=\mu$,
which follows from $\{\mu_{n_i}\}\rightarrow \mu$, we have
\begin{equation}\label{seq2}
\lim_{i\rightarrow \infty}\int f_k \, d\mu^{\prime}_{n_i}=\int f_k
\, d\mu.
\end{equation}
Replacing $n$ by $n_i$ in \eqref{seqlem2}, dividing by $n_i$ and
passing to $\limsup_{i\rightarrow \infty}$,  \eqref{seq1} follows by
\eqref{seq2}. Letting $k\rightarrow \infty$, the result holds.
\end{proof}

\begin{theorem}
Let $T$ be a continuous bundle RDS on $\mathcal{E}$,
$\Phi=\{f_n\}_{n=1}^{\infty}$ is subadditive, $f_1\in
\mathbf{L}_{\mathcal{E}}^1(\Omega, \mathcal{C}(X))$, and each $f_n$
be measurable in $\omega$ and continuous in $x$. If
$\pi_T(\Phi)>-\infty$, then
\begin{equation*}
\pi_T(\Phi)=\sup\{  h_{\mu}^{(r)}(T) + \Phi^*(\mu): \mu \in
\mathcal{M}_{\mathbf{P}}^1(\mathcal{E},T)\, \text{and}
\,\,\Phi^*(\mu)>-\infty \}.
\end{equation*}
\end{theorem}

\begin{proof}
 Let $\mu \in \mathcal{M}_{\mathbf{P}}^1(\mathcal{E},T) $,
 $\Phi^*(\mu)>-\infty$, $\mathcal{P}=\{P_1, \cdots,
 P_k\}$, be a finite measurable partition of $X$, and $\epsilon$ be
 a  positive number with $\epsilon k \log k <1$.
 Denote by $\mathcal{P}(\omega)=\{P_1(\omega), \cdots,
 P_k(\omega)\} \}$, $P_i(\omega)=P_i \cap \mathcal{E}_{\omega}, i=1,\cdots,
 k$, the corresponding partition of $\mathcal{E}_{\omega}$. By the
 regularity of $\mu$, we can find compact sets $Q_i \subset P_i, 1 \leq i \leq
 k$, such that
 \begin{equation*}
\mu (P_i \backslash Q_i)= \int \mu_{\omega}(P_i(\omega)\backslash
Q_i(\omega)) \, d\mathbf{P}(\omega) < \epsilon,
 \end{equation*}
where $Q_i(\omega)= Q_i \cap \mathcal{E}_{\omega}$. Let
$\mathcal{Q}(\omega)=\{Q_0(\omega),\cdots, Q_k(\omega) \}$ be the
partition of $\mathcal{E}_{\omega}$, where $Q_0(\omega)=
\mathcal{E}_{\omega} \backslash \bigcup_{i=1}^kQ_i(\omega)$. Then by
the results of Kifer \cite{Kifer} and Bogensch{\"u}tz \cite{Bogen},
the following inequality holds (See \cite{Kifer2001} for details),
\begin{equation}\label{entr}
h_{\mu}^{(r)}(T, \Omega \times \mathcal{P})\leq  h_{\mu}^{(r)}(T,
\Omega \times \mathcal{Q}) +1,
\end{equation}
where $\Omega \times \mathcal{P}$ (respectively by $\Omega \times
\mathcal{Q}$) denotes the partition of $ \Omega \times X$ into sets
$\Omega \times P_i$ (respectively by $\Omega \times Q_i$). Set
\begin{equation*}
\mathcal{Q}_n(\omega)=
\bigvee_{i=0}^{n-1}(T^i_{\omega})^{-1}\mathcal{Q}(\vartheta^i\omega),
\quad  f^*_n(\omega, C):=\sup\{f_n(\omega,x): x\in \overline{C} \}
\end{equation*}
for $C \in \mathcal{Q}_n(\omega)$. Then  by the well-known
inequality (See \cite{Walters})  $ \sum_{1\leq i\leq m}p_i(a_i- \log
p_i)\leq \sum_{1\leq i\leq m} \exp a_i, $  where  each $a_i$ is a
real number, $p_i\geq 0 $ and $\sum_{i=1}^k p_i =1$, we have
\begin{equation}\label{pre1}
 H_{\mu_{\omega}}(\mathcal{Q}_n(\omega))+\int_{\mathcal{E}_{\omega}}f_n(\omega)
 \,d\mu_{\omega} \leq \log \sum_{C\in \mathcal{Q}_n(\omega)} \exp
 f_n^*(\omega,C).
\end{equation}
Let $ \mathcal{R}=\{Q_0 \cup Q_1, \cdots,  Q_0 \cup Q_k\}$ be the
open cover set of $X$, and $\delta$ be the Lebesgue number for
$\mathcal{R}$. Then for every $\omega$,
$\mathcal{R}_{\omega}=\{Q_0(\omega) \cup Q_1(\omega), \cdots,
Q_0(\omega) \cup Q_k(\omega) \}$ is the open cover of
$\mathcal{E}_{\omega}$ and $\delta$ is also a Lebesgue number. Let
$x(C)$ be the point in $\overline{C}$ such that $f_n(\omega,x(C) )=
f_n^*(\omega, C)$. If $d(x(C),x(D))<\delta$, then $x(C)$ and $x(D)$
are in the same element of $\mathcal{R}_{\omega}$, say $Q_0(\omega)
\cup Q_j(\omega), 0\leq j <k+1$. Hence for each $C$, there are at
most $2^n$ elements $D$ of $\mathcal{Q}_n(\omega)$ such that
\begin{equation*}
d_n^{\omega}(x(C),x(D))= \max_{0\leq j<
n}\{d(T_{\omega}^j(x(C)),T_{\omega}^j(x(D)) \}<\delta.
\end{equation*}
Now an $(\omega, \delta, n)$-separated set $E$ can be constructed
such that
\begin{equation}\label{card}
 \sum_{C\in \mathcal{Q}_n(\omega)} f_n^*(\omega, C)\leq 2^n \sum_{y\in
 E}f_n(\omega,y)
\end{equation}
We first select the point $x(C_1)$ such that $f_n^*(\omega,
C_1)=\max_{C\in \mathcal{Q}_n(\omega)}f_n^*(\omega, C)$, then select
the second point $x(C_2)$ such that
$$
f_n^*(\omega, C_2)=\max_{\substack{C^{\prime}\in \mathcal{Q}_n(\omega)\\
 d_n^{\omega}(x(C_1), x(C^{\prime})) \geq \delta }} f_n^*(\omega,
 C^{\prime}),
$$
the third point $x(C_3)$ such that
$$
f_n^*(\omega, C_3)=\max_{\substack{C^{\prime\prime}\in \mathcal{Q}_n(\omega)\\
 d_n^{\omega}(x(C_1), x(C^{\prime\prime})) \geq \delta \\
  d_n^{\omega}(x(C_2), x(C^{\prime\prime}))\geq \delta}}
  f_n^*(\omega, C^{\prime\prime}),
$$
continue this process, a finite step $m$ can complete this selection
since $\mathcal{Q}_n(\omega)$ is finite. Let $E=\{x(C_1), \cdots,
x(C_m) \}$. Obviously  $E$ is an $\omega, \delta,n$-separated set.
By the above analysis, for each step, we delete at most $2^n$
elements of $\mathcal{Q}_n(\omega)$, so the inequality \eqref{card}
holds.

From \eqref{pre1} and \eqref{card}, we get
\begin{equation*}
H_{\mu_{\omega}}(\mathcal{Q}_n(\omega))+\int_{\mathcal{E}_{\omega}}f_n(\omega)
 \,d\mu_{\omega} \leq
 n\log2 +\log \pi_T(\Phi)(\omega, \delta, n).
\end{equation*}
Integrating this inequality through $\mathbf{P}$, dividing by $n$
and letting $n\rightarrow \infty$, from \eqref{pressure},
\eqref{entr} and $\Phi^*(\mu)>-\infty$, we have
\begin{align*}
&H_{\mu}^{(r)}(T, \Omega \times \mathcal{P}) + \Phi^*(\mu)\\
 \leq
1+H_{\mu}^{(r)}(&T, \Omega \times \mathcal{Q})\Phi^*(\mu)
 \leq  1+\log 2 +
\pi_T(\Phi)(\delta).
\end{align*}
By the arbitrariness of $\mathcal{P}$ and $\delta$,
\begin{equation*}
h_{\mu}^{(r)}(T) + \Phi^*(\mu) \leq 1+\log 2 + \pi_T(\Phi).
\end{equation*}
Replacing $T$ and $\Phi$ by $T^n$ and $\Phi^n$, respectively, then
by the equality $h_{\mu}^{(r)}(T^n)=n h_{\mu}^{(r)}(T)$ (See
\cite{Bogen,Kifer}) and $(\Phi^n)^*(\mu)= n \Phi^*(\mu)$, we obtain
\begin{equation*}
n(h_{\mu}^{(r)}(T) + \Phi^*(\mu)) \leq 1+\log 2 + \pi_{T^n}(\Phi^n).
\end{equation*}
Using lemma \ref{lem1}, dividing by $ n $ and letting $n\rightarrow
\infty$, we get the fist part
\begin{equation*}
 h_{\mu}^{(r)}(T) + \Phi^*(\mu)) \leq  \pi_{T }(\Phi ).
\end{equation*}

In the opposite direction, choose some small $\epsilon
> 0$ with $\pi_T(\Phi,\epsilon)>-\infty$,  and a family of
measurable in $\omega$ maximal $(\omega, \epsilon, n)$-separated
sets $G(\omega, \epsilon, n) \subset \mathcal{E}_{\omega}$  by Lemma
\ref{lemkifer} such that
\begin{equation}\label{sepapre}
 \sum_{x\in G(\omega, \epsilon, n)} \exp f_n(\omega,x) \geq
 \frac{1}{e} \pi_T(\Phi)(\omega, \epsilon, n).
\end{equation}
Let $\{\nu_{\omega}^{(n)}\} $ be a family of atomic measures on
$\mathcal{E}_{\omega}$ such that they are measurable disintegrations
of some probability measure $\nu^{(n)}$, i.e.,
$d\nu^{(n)}(\omega,x)=d\nu^{(n)}_{\omega}(x)\,d\mathbf{P}(\omega)$,
where
\begin{equation*}
\nu_{\omega}^{(n)}=\frac{\sum_{x\in G(\omega, \epsilon, n)}\exp
f_n(\omega,x)\delta_x}{\sum_{y\in G(\omega, \epsilon, n)}\exp
f_n(\omega,y)}.
\end{equation*}
Denote
$$
\mu^{(n)}=\frac{1}{n}\sum_{i=0}^{n-1}\Theta^i \nu^{(n)}.
$$
By \eqref{pressure} and Lemma 2.1 (i)-(ii) in \cite{Kifer}, choose a
subsequence $\{n_j\}$  satisfying the following two limits
simultaneously,
\begin{align}\label{twolim}
 \pi_T(\Phi)(\epsilon)&= \lim_{j\rightarrow
\infty} \frac{1}{n_j} \int \log \pi_T(\Phi)(\omega,\epsilon, n_j) \,
d\mathbf{P}(\omega), \nonumber \\
 \lim_{j\rightarrow \infty} \mu^{(n_j)}&= \mu \quad \text{for some}
 \quad \mu \in \mathcal{M}_{\mathbf{P}}^1(\mathcal{E},T).
\end{align}
Choose a partition $\mathcal{P}=\{P_1, \cdots, P_k\}$ of $X$ with
diam $\mathcal{P} \leq \epsilon$ and $\int \mu_{\omega}(\partial
P_i) \, d\mathbf{P}(\omega)=0$ for all $1\leq i \leq k$, where
$\partial$ denotes the boundary. Set
$\mathcal{P}(\omega)=\{P_1(\omega), \cdots, P_k(\omega)\}$,
$P_i(\omega)=P_i \cap \mathcal{E}_{\omega}$. Since each element of
$\bigvee_{i=0}^{n-1}(T^i_{\omega})^{-1}\mathcal{P}(\vartheta^i\omega)$
contains at most one element of $G(\omega, \epsilon, n)$, by
\eqref{sepapre}, we have
\begin{align*}
\begin{split}
&H_{\nu_{\omega}^{(n)}}\bigl(\bigvee_{i=0}^{n-1}(T^i_{\omega})^{-1}\mathcal{P}(\vartheta^i\omega)
\bigr) + \int f_n(\omega)\, d\nu_{\omega}^{(n)}\\
=\log \bigl( &\sum_{x\in G(\omega, \epsilon, n)}\exp
f_n(\omega,x)\bigr)  \geq  \log \pi_T(\Phi)(\omega, \epsilon, n)-1.
\end{split}
\end{align*}
Let $\mathcal{Q}=\{Q_1, \cdots, Q_k\}$, $Q_i=(\Omega \times P_i)\cap
\mathcal{E}$; then $\mathcal{Q}$ is a partition of $\mathcal{E}$ and
$Q_i(\omega)=\{x\in \mathcal{E}_{\omega}:(\omega,x)\in
Q_i\}=P_i(\omega)$. Integrating the inequality against $\mathbf{P}$,
by the definition of the conditional entropy, we get
\begin{equation*}%\label{seq3}
H_{\nu^{(n)}}\bigl(
\bigvee_{i=0}^{n-1}(\Theta^i)^{-1}\mathcal{Q}\mid\mathcal{F}_{\omega}
\bigr) + \int f_n(\omega,x) \, d\nu^{(n)} \geq \int \log
\pi_T(\Phi)(\omega, \epsilon, n) \, d\mathbf{P}(\omega)-1.
\end{equation*}
Setting $q, n \in \mathbb{N}$, $1<q<n$, using the usual method as in
\cite{Walters} and using the subadditivity of conditional entropy
\cite{Kifer,Bogenthesis} and Lemma 3.2 in \cite{LedWal}, we have the
following inequality (  See \cite{Kifer2001} for  the detail)
$$
q H_{\nu^{(n)}}\bigl(
\bigvee_{m=0}^{n-1}(\Theta^i)^{-1}\mathcal{Q}\mid\mathcal{F}_{\omega}
\bigr) \leq n
H_{\mu^{(n)}}\bigl(\bigvee_{i=0}^{q-1}(\Theta^i)^{-1}\mathcal{Q}\mid\mathcal{F}_{\omega}\big)
+2q^2\log k.
$$
Then by the above two inequalities,
\begin{align*}
&q\int \log \pi_T(\Phi)(\omega, \epsilon, n) \,
d\mathbf{P}(\omega)-q
\\ \leq n
H_{\mu^{(n)}}\bigl(&\bigvee_{i=0}^{q-1}(\Theta^i)^{-1}\mathcal{Q}\mid\mathcal{F}_{\omega}\big)
+2q^2\log k + \int f_n(\omega,x)\, d\nu^{(n)}.
\end{align*}
Since $\mu \in
 \mathcal{M}_{\mathbf{P}}^1(\omega,T)$ and
$\bigvee_{i=0}^{q-1}(T_{\omega}^i)^{-1}\mathcal{P}(\vartheta^i\omega)
\subset \bigcup (T_{\omega}^i)^{-1}\mathcal{P}(\vartheta^i\omega)$,
it is easy to see that $\mu_{\omega}( \partial
\bigvee_{i=0}^{q-1}(T_{\omega}^i)^{-1}\mathcal{P}(\vartheta^i\omega))=0
\, \mathbf{P}$-a.s.
 Dividing by $n$, passing to the limit along a
subsequence $n_j\rightarrow \infty $ satisfying \eqref{twolim} and
taking into account Lemma \ref{lem2} and Lemma 2.1 (iii) in
\cite{Kifer}, it follows in view of the choice of the partition
$\mathcal{P}$ that
$$
q\pi_T(\Phi)(\epsilon)\leq
H_{\mu}\bigl(\bigvee_{i=0}^{q-1}(\Theta^i)^{-1}\mathcal{Q}\mid\mathcal{F}_{\omega}\big)+
q \Phi^*(\mu),
$$
then $\Phi^*(\mu)>-\infty$ since $\pi_T(\Phi)>-\infty$. Dividing
this inequality by $q$ and letting $q\rightarrow \infty$, so
$\pi_T(\Phi)(\epsilon)\leq h_{\mu}^{(r)}(T,\mathcal{Q})+
\Phi^*(\mu)$. Hence $\pi_T(\Phi)(\epsilon)\leq h_{\mu}^{(r)}(T)+
\Phi^*(\mu)$ and letting $\epsilon \rightarrow 0$, the required
inequality follows, then the results holds.
\end{proof}

\begin{remark}
If $\pi_T(\Phi)\geq\sup\{  h_{\mu}^{(r)}(T) + \Phi^*(\mu)\}$, then
obviously $\pi_T(\Phi)>-\infty$ by $\Phi^*(\mu)>-\infty$. So the
condition we give is only used in the opposite direction $``\leq"$.
In fact, by the above argument, it is not hard to see that
$\pi_T(\Phi)=-\infty$ is equivalent to $\Phi^*(\mu)=-\infty$ for all
invariant measure $\mu$.
\end{remark}

\section*{Acknowledgements}
The first author is supported by a grant from Postdoctoral Science
 Research Program of Jiangsu Province (0701049C).
 The second  author is partially supported by the National Natural Science Foundation of China (10571086)
 and National Basic Research Program of China (973 Program) (2007CB814800).

\end{document}